\documentclass[
final
]{dmtcs-episciences}

\usepackage[usenames]{color}
\usepackage{amssymb}
\usepackage{amsmath}
\usepackage{amsthm}
\usepackage{amsfonts}
\usepackage{amscd}
\usepackage{lscape}
\usepackage[title]{appendix}
\usepackage{slashbox}
\usepackage{color}
\usepackage{float}
\usepackage{graphics}
\usepackage{latexsym}
\usepackage{breakurl}
\newcommand{\seqnum}[1]{\href{https://oeis.org/#1}{\rm \underline{#1}}}
\def\suchthat{\, : \, }

\DeclareMathOperator{\neqtriple}{NeqTriple}
\DeclareMathOperator{\shevcond}{ShevCond}
\DeclareMathOperator{\poweroftwo}{Power2}
\DeclareMathOperator{\triple}{Triple}
\DeclareMathOperator{\mwtriple}{TripleMW}
\DeclareMathOperator{\vtmtriple}{TripleVTM}
\DeclareMathOperator{\pdtriple}{TriplePD}
\DeclareMathOperator{\twoconsec}{twoconsec}
\DeclareMathOperator{\sep}{sep}
\def\land{\, \wedge\, }
\def\lor{\, \vee\, }
\def\modd#1 #2{#1\ \mbox{\rm (mod}\ #2\mbox{\rm )}}

\def\lcm{{\rm lcm}}

\newcommand{\pseudoperiod}{\texttt{PSEUDOPERIOD}}
\newcommand{\hittingset}{\texttt{HITTING SET}}

\usepackage[utf8]{inputenc}
\usepackage{subfigure}

\usepackage[round]{natbib}

\author[Joseph Meleshko et al.]{Joseph Meleshko\affiliationmark{1}
  \and Pascal Ochem\affiliationmark{2}
  \and Jeffrey Shallit\affiliationmark{1}\thanks{The research of JS is supported by NSERC Grant 2018-04118.}
  \and Sonja Linghui Shan\affiliationmark{1}}
\title[Pseudoperiodic Words]{Pseudoperiodic Words and a Question of Shevelev}
\affiliation{
  University of Waterloo, Canada \\
  LIRMM, CNRS, Universit\'e de Montpellier, France}
\begin{document}

\publicationdata{vol. 25:2 }{2023}{6}{10.46298/dmtcs.9919}{2022-08-15; 2022-08-15; 2023-01-02}{2023-05-26}

\maketitle
\begin{abstract}
We generalize the familiar notion of periodicity in sequences to a new kind of pseudoperiodicity, and we prove some basic results about it.
We revisit the results of a 2012 paper of Shevelev and reprove his results in a simpler and more unified manner, and provide a complete answer to one of his previously unresolved questions.  We consider finding words 
with specific pseudoperiod and having the smallest possible critical exponent.
Finally, we consider the problem of determining whether a finite
word is pseudoperiodic of a given size, and show that it is NP-complete.   
\end{abstract}

\theoremstyle{plain}
\newtheorem{theorem}{Theorem}
\newtheorem{corollary}[theorem]{Corollary}
\newtheorem{lemma}[theorem]{Lemma}
\newtheorem{proposition}[theorem]{Proposition}

\theoremstyle{definition}
\newtheorem{definition}[theorem]{Definition}
\newtheorem{example}[theorem]{Example}
\newtheorem{conjecture}[theorem]{Conjecture}
\newtheorem{openproblem}[theorem]{Open Problem}
\newtheorem{problem}[theorem]{Problem}

\theoremstyle{remark}
\newtheorem{remark}[theorem]{Remark}

\title{Pseudoperiodic Words and a Question of Shevelev}

\centerline{\emph{In honor of Vladimir Shevelev (1945--2018)}}

\section{Introduction}
Periodicity is one of the simplest and most studied aspects of
words (sequences).  Let $w = a_0 a_1 a_2 \cdots a_{t-1}$ 
be a finite word.  We say that $w$ is (purely) \emph{periodic} with
period $p$ ($1 \leq p \leq t$) if
$a_i = a_{i+p}$ for $0 \leq i < t-p$.   For example, the
French word \texttt{entente} is periodic with
periods $3, 6, $ and $7$.  The definition is
extended to infinite words as follows:  ${\bf w} = a_0 a_1 \cdots$ is
periodic with period $p$ if $a_i = a_{i+p}$ for all $i \geq 0$.   
Unless otherwise stated,
all words in this paper are indexed starting with index $0$.
All infinite words are defined over a finite alphabet.

In this paper we begin the study of a simple and obvious---yet apparently little-studied---generalization of
periodicity, which we call $k$-pseudoperiodicity.  

\begin{definition}
We say that
a finite word $w = a_0 a_1 \cdots a_{t-1}$ is \emph{$k$-pseudoperiodic} if there exist
$k \geq 1$ integers $0 < p_1 < p_2 < \cdots < p_k$ such that
$a_i \in \{ a_{i+p_1}, a_{i+p_2}, \ldots, a_{i+p_k} \}$
for all $i$ with $0 \leq i < t-p_k$.
For infinite words the 
membership must hold for all $i$.  If this is the case, we call
$( p_1, p_2, \ldots, p_k )$ a \emph{pseudoperiod} for
$w$.
\end{definition}
In this paper, when we write a pseudoperiod $( p_1, p_2, \ldots, p_k )$  we always assume $0<p_1 < \cdots < p_k$.   Note that $1$-pseudoperiodicity is the ordinary notion of (pure) periodicity.   
If an infinite word $\bf w$ is $k$-pseudoperiodic for
some $k < \infty$, we call it \emph{pseudoperiodic}.

We note that our definition of pseudoperiodicity is not the same as that
studied by \cite{BlondinMasse&Gaboury&Halle:2012}.
Nor is it the same as the notion of
quasiperiodicity, as introduced by
\cite{Marcus:2004}, and now
widely studied in many papers.  Nor is it the same as ``almost periodicity'', which is more commonly called uniform recurrence (i.e., every block that occurs, occurs with bounded gaps between successive occurrences).

\subsection{Notation}
\label{notation}

We use the familiar regular expression notation for regular languages.
For infinite words, we let
$x^\omega$ for a nonempty finite word $x$ denote the infinite word $xxx\cdots$.

The \emph{exponent} of a finite word $x$, denoted $\exp(x)$ is
$|x|/p$, where $p$ is the smallest period of $x$.
For example, if $x = \texttt{entente}$, then $\exp(x) = 7/3$.
If $q$ divides $|x|$, then by $x^{p/q}$ we mean
the word of length $p|x|/q$ that is a prefix of $x^\omega$.
For example, ${\texttt{(alf)}}^{7/3} = \texttt{alfalfa}$.

If all the nonempty factors $f$ of a (finite or infinite) word $x$ satisfy
$\exp(f) < e$, we say that $x$ is \emph{$e$-free}.  If they satisfy $\exp(f) \leq e$, we say that
$x$ is \emph{$e^+$-free}.

The \emph{critical exponent} of an infinite word $\bf w$ is the supremum
of $\exp(x)$ over all finite nonempty factors $x$ of $\bf w$.
Here the supremum is taken over the extended real numbers, where for each real number $\alpha$
there is a corresponding number $\alpha^+$ satisfying $\alpha < \alpha^+ < \beta$
for all $\beta > \alpha$.
Thus if $x$ is a real number, the inequality $x \geq \alpha^+$
has the same meaning as $x > \alpha$.

If $S$ is a set of (finite or infinite) words, then its \emph{repetition threshold}
is the infimum of the critical exponents of all its words.

A \emph{run} in a word is a maximum block of consecutive identical letters.
The first run is called the \emph{initial run}.

An \emph{occurrence} of a finite nonempty word $x$ in another word $w$
(finite or infinite) is an index $i$ such that $w[i+j]=x[j]$ for $0 \leq j < |x|$.
The  \emph{distance} between two occurrences $i$ and $i'$ is their difference $|i'-i|$.

The \emph{Thue-Morse word} ${\bf t} = 01101001 \cdots$ is the infinite
fixed point, starting with $0$, of the morphism $\mu(0) = 01$ and $\mu(1) = 10$.

\subsection{Goals of this paper}

There are five basic questions that interest us in this paper. 
\begin{enumerate}
    \item Given an infinite sequence
    $\bf s$, is it pseudoperiodic?
    \item If $\bf s$ is $k$-pseudoperiodic for some $k$, what
    is the smallest such $k$?
    \item If $\bf s$ is $k$-pseudoperiodic, what  are all the possible 
    pseudoperiods of size $k$?
    \item What is the smallest possible critical exponent of
    an infinite pseudoperiodic
    word with specified pseudoperiod?
    \item How quickly can we tell if a given
    finite sequence has a pseudoperiod of bounded size?
\end{enumerate}

In particular, we are interested in answering these questions for the class of sequences called
automatic.
A novel feature of our work is that much of it is done
using a theorem-prover for automatic sequences,
called \texttt{Walnut}, originally developed by Hamoon Mousavi.
For more information about \texttt{Walnut}, see \cite{Mousavi:2016,Shallit:2022}.

Here is a brief summary of what we do in our paper.  In Section~\ref{sec2} we prove basic results about pseudoperiodicity, and show that questions 1, 2, and 3 above are decidable for the class of automatic sequences.  In Section~\ref{sec3}, we recall Shevelev's problems about pseudoperiods of the Thue-Morse word, solve them using our method, and also solve his open question from 2012.  In Section~\ref{sec4}, we obtain analogous pseudoperiodicity results for some other famous sequences.   In Section~\ref{sec5} we turn to question 4, obtaining the best possible critical exponent for binary words having certain pseudoperiods.  In Section~\ref{sec6} we treat the case of larger alphabets and obtain some results.  In Section~\ref{sec7} we prove that checking the existence of a pseudoperiod of size $k$ is, in general, a difficult computational problem, thus answering question 5.  Along the way, we state two conjectures (Conjectures \ref{conj1} and \ref{conj2}) and one open problem (Open Problem \ref{open1}).  Finally, in Section~\ref{sec8}, we make some brief biographical remarks about Vladmir Shevelev.  

\section{Basic results}
\label{sec2}

\begin{proposition}
An infinite word $\bf s$ is pseudoperiodic if and only if there exists a bound $B< \infty$ such that
two consecutive occurrences of the same letter in $\bf s$
are always separated by distance at most $B$.
\label{proptwo}
\end{proposition}

\begin{proof}
Suppose $\bf s$ has pseudoperiod $(p_1, p_2, \ldots, p_k)$, with
$p_1 < \cdots < p_k$.  Then clearly we may take $B = p_k$.

On the other hand, if two consecutive occurrence of every letter 
are always separated by distance $\leq B$, then
we may take $(1,2,\ldots, B)$ as a pseudoperiod for $\bf s$.
\end{proof}

For binary words we can say this in another way.
\begin{proposition}
Let $x$ be an infinite binary word.
\begin{enumerate}
    \item[(a)] If $M$ is the maximum element of a pseudoperiod,
    then the longest non-initial run in $\bf x$ is of length $\leq M-1$;
    \item[(b)] if the longest non-initial run length in $\bf x$ is $B$,
    then $(1,2,3, \ldots, B+1)$ is a pseudoperiod.
\end{enumerate}
In particular, an infinite binary word is pseudoperiodic if and only if it consists of
a single letter repeated, or its sequence of run lengths is bounded.
\label{prop1}
\end{proposition}

\begin{proof}
Suppose $\bf x$ is pseudoperiodic with pseudoperiod
$(p_1, p_2, \ldots, p_k)$ and let $M = \max_{1 \leq i \leq k} p_i$.  Let $a \in \{ 0, 1 \}$ and let
${\bf x}[p..q]$ be a run of $a$'s and
${\bf x}[q+1..r]$ be the following run (of $\overline{a}$'s).
Then ${\bf x}[r+1] = a$.  Now consider ${\bf x}[q] = a$.
Since $\bf x$ is pseudoperiodic, we know that
$(r+1)-q \leq M$.  Hence all non-initial runs are of
length at most $M-1$.  

On the other hand, if index $p$
does not correspond to the last
letter of a run, then ${\bf x}[p] = {\bf x}[p+1]$.  If it does so correspond, since the word is binary
and
all non-initial run lengths are bounded, say by $B$,
we know that ${\bf x}[p+i] = {\bf x}[p]$ for some
$i \leq B+1$.   So $(1,2,\ldots, B+1)$ is a pseudoperiod.  
\end{proof}

\begin{proposition}
The only infinite words with pseudoperiod $(1,2)$ are those of the form
$a^\omega$ or $a^* (ab)^\omega$ and $b^*(ba)^\omega$ for distinct letters $a,b$.  
The only finite words with pseudoperiod $(1,2)$
are those of the form $a^*(ab)^*(a+\epsilon)$ with $a \not= b$.
\label{pp12}
\end{proposition}

\begin{proof}
Follows immediately from Proposition~\ref{prop1}.
\end{proof}

\begin{theorem}
If an infinite word has pseudoperiod $S$
then it has $\leq \max S$ distinct letters.
If it has exactly $\max S$ distinct letters, then it must have a suffix of
the form $x^\omega$,
where $x$ is a word of length $\max S$ containing each letter exactly once.
\end{theorem}

\begin{proof}
Suppose $\bf w$ has pseudoperiod $S$, with $k = \max S$.
Since each occurrence of a letter is followed by another occurrence of the same letter
at distance $\leq k$, it follows that each letter of $\bf w$ must occur with
frequency $\geq 1/k$ in $w$. But the total of
all frequencies must sum to $1$, so there cannot be more than $k$ distinct letters.

Now suppose $\bf w$ has exactly $k$ distinct letters, say $0, 1, \ldots, k-1$.
Without loss of generality, assume that the last letter to occur for the first time is $k-1$ and
$p_{k-1}$ is this first occurrence.
Furthermore, let $p_0, \ldots, p_{k-2}$ be the positions of the last
occurrence of the letters $0, 1, \ldots, k-2$ that precede
$p_{k-1}$, and again, without loss of generality assume
$p_0 < \cdots < p_{k-2} < p_{k-1}$.  
Thus ${\bf w}[p_0..p_{k-1}] = 0 \, w_1\, 1\, w_2\, 2 \cdots (k-2) \, w_{k-1} \, (k-1)$ for some words
$w_1, w_2, \ldots, w_{k-1}$, where $w_i$ contains
no occurrences of letters $<i$.  
However, if any of these $w_i$ were nonempty then $\bf w$ could not be
pseudoperiodic (because the $0$ at position $p_0$ would not be followed by
another $0$ at distance $\leq k$).
So all the $w_i$ are  empty. Furthermore, pseudoperiodicity also shows that
${\bf w}[p_{k-1}+1] = 0$, and inductively, that ${\bf w}[p_{k-1} + i] =
(i-1) \bmod k$ for all $i \geq 0$.
\end{proof}

We now turn to results about automatic sequences.
This is a large and interesting class of sequences where the $n$th term
is computed by a finite automaton taking as input the representation of $n$
in some base (or generalizations, such as Fibonacci base).
For more information about automatic sequences, see \cite{Allouche&Shallit:2003}.

\begin{corollary}
   Problems 1, 2, and 3 above are decidable, if
   $\bf s$ is an automatic sequence.
\end{corollary}

\begin{proof}
By the results of \cite{Bruyere&Hansel&Michaux&Villemaire:1994},
it suffices to create first-order logical formulas asserting each property.
The domain of the variables in all logical statements is assumed to be
$\mathbb{N} = \{0,1,2,\ldots \}$, the natural numbers.

By Proposition~\ref{proptwo}, we know that $\bf s$ is pseudoperiodic
if there is a bound on the separation of two consecutive occurrences of the same letter. 
We can assert this as follows.
First, define a formula that asserts that $i < j$ are two
consecutive occurrences of the same letter:
$$
\twoconsec(i,j) := 
   i<j \, \wedge\, {\bf s}[i] = {\bf s}[j] \, 
   \wedge\, \forall p \ (p>i \, \wedge\, p<j) \implies
   {\bf s}[i] \not= {\bf s}[p] .
$$
Next, the formula
$$\sep(B) :=  \forall i, j \ \twoconsec(i,j) \implies j \leq i+B .$$
asserts the claim that two consecutive occurrences of the same letter
are separated by at most $B$.
Finally, the formula $\exists B\ \sep(B)$ evaluates to
\texttt{TRUE} if and only $\bf s$ is pseudoperiodic.
This solves the first problem.

Once we know that $\bf s$ is pseudoperiodic, we can
find the smallest $B$ 
such that $\sep(B)$ holds.   To do so, form the automaton for
$$ \sep(B) \, \wedge\, \neg\sep(B-1); $$
it will accept exactly one value of $B$, which
is the desired minimum.
This tells us
that ${\bf s}$ has pseudoperiod $(1,2,\ldots, B)$,
so certainly it is $B$-pseudoperiodic.  

We can now write the assertion that $\bf s$ has
a pseudoperiod of size $p$, as follows:
\begin{multline*}
\exists a_1, a_2, \ldots, a_p \ 
1\leq a_1 \, \wedge\, a_1 < a_2 \, \wedge\, \cdots \,\wedge\, a_{p-1} < a_p \, \wedge\,  \\
\forall n \ 
({\bf s}[n]={\bf s}[n+a_1] \, \vee\,
{\bf s}[n]={\bf s}[n+a_2] \, \vee\, \cdots \, \vee
{\bf s}[n]={\bf s}[n+a_p] ) .
\end{multline*}
By testing this for $p = 1, \ldots, B$, we can
find the smallest $p$ for which this holds.
This solves problem 2.

Finally, we can determine all possible pseudoperiods
of size $p$ with the formula 
\begin{multline*}
    1\leq a_1 \, \wedge\, a_1 < a_2 \, \wedge\, \cdots \, \wedge\, a_{p-1} < a_p \, \wedge\,  \\
\forall n \ 
({\bf s}[n]={\bf s}[n+a_1] \, \vee\,
{\bf s}[n]={\bf s}[n+a_2] \, \vee\, \cdots \, \vee
{\bf s}[n]={\bf s}[n+a_p] ) .
\end{multline*}
The corresponding finite automaton accepts all the possible
pseudoperiods $(a_1, \ldots, a_p)$ of size $p$.
\end{proof}

From these ideas we can prove an interesting corollary.

\begin{corollary}
Suppose the automatic sequence $\bf s$ is {\rm not\/} $k$-pseudoperiodic.  
Then there exists a constant $C$ (depending only on $\bf s$) such that for all
$k$-tuples $0< p_1 < p_2 < \cdots < p_k$,
the smallest $n$ for which
${\bf s}[n] \not\in \{ {\bf s}[n+p_1], {\bf s}[n+p_2], \ldots,
{\bf s}[n+p_k] \}$ satisfies $n \leq Cp_k$.
\label{linear-bound}
\end{corollary}

\begin{proof}
A trivial variation on the previous arguments shows that if $\bf s$
is automatic, then there is an automaton accepting, in parallel,
$n, p_1, p_2, \ldots, p_k$ such that $n$ is the smallest natural
number satisfying
${\bf s}[n] \not\in \{ {\bf s}[n+p_1], {\bf s}[n+p_2], \ldots,
{\bf s}[n+p_k] \}$.   Thus, in the terminology of
\cite{Shallit:2021h}, this $n$ can be considered a ``synchronized
function'' of $(p_1, \ldots, p_k)$.  We can then apply the known
linear bound on synchronized functions \cite[Thm.~8]{Shallit:2021h}
to deduce the existence
of $C$ such that $n \leq C p_k$.
\end{proof}

Although, as we have just seen, these problems are all decidable for automatic sequences in theory, in practice, the automata that result can be extremely large and require a lot of computation to find.
We can use \texttt{Walnut}, a theorem-prover originally designed by 
\cite{Mousavi:2016} to translate logical formulas to automata.   

\begin{example}
Let us consider an example, the \emph{Fibonacci word} 
${\bf f} = 01001010 \cdots$, the fixed point of the
morphism $0 \rightarrow 01$, $1 \rightarrow 0$.
The following \texttt{Walnut} code demonstrates
that it is $2$-pseudoperiodic.  (In fact, this follows
from the much more general Proposition~\ref{prop5}
below.)
\begin{verbatim}
eval isfibpseudo "?msd_fib Ea,b 1<=a & a<b & 
   An (F[n]=F[n+a]|F[n]=F[n+b])":
\end{verbatim}
It returns \texttt{TRUE}.

We can determine all possible pseudoperiods of size $2$ using 
\texttt{Walnut}, as follows:
\begin{verbatim}
def fib2pseudoperiod "?msd_fib 1<=a & a<b & 
   An (F[n]=F[n+a]|F[n]=F[n+b])":
\end{verbatim}
The resulting automaton accepts all pairs $(a,b)$
that are pseudoperiods of $\bf f$, in Fibonacci
representation.  It has 28 states and is displayed in 
Figure~\ref{fig5}.
\begin{figure}[htb]
\begin{center}
    \includegraphics[width=5.5in]{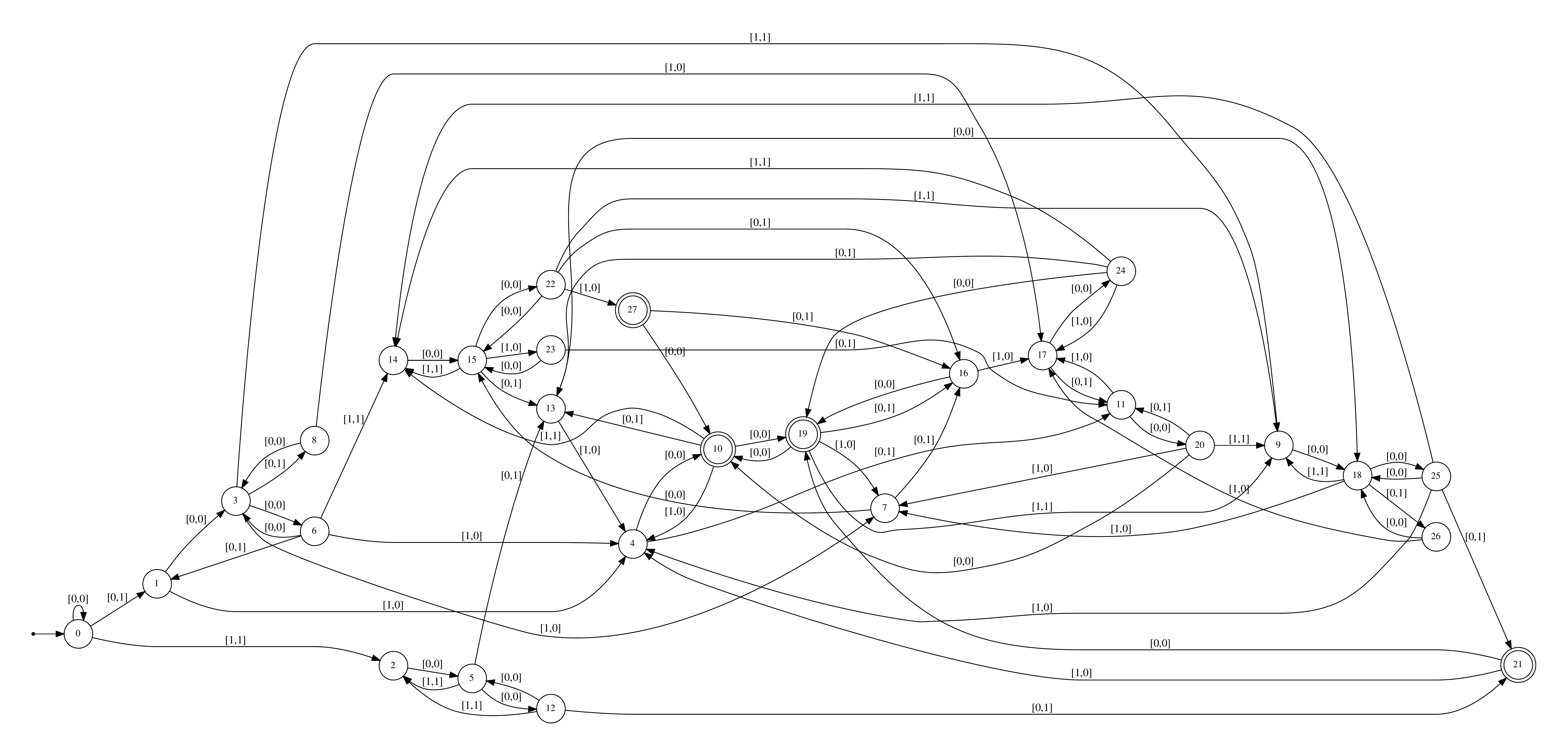}
\end{center}
\caption{Pseudoperiods of size $2$ for the Fibonacci word.}
\label{fig5}
\end{figure}
\end{example}

We now consider a famous uncountable class of binary sequences, the \emph{Sturmian words} \cite[Chap.~2]{Lothaire:2002}.   These are infinite words
of the form ${\bf s}_{\alpha, \beta} := (\lfloor (n+1) \alpha + \beta \rfloor - \lfloor n \alpha + \beta \rfloor)_{n \geq 1}$, where $0 < \alpha < 1$ is an irrational
real number and $0 \leq \beta < 1$.

\begin{proposition}
Every Sturmian sequence is $2$-pseudoperiodic but not $1$-pseudoperiodic.
\label{prop5}
\end{proposition}

\begin{proof}
If ${\bf s}_{\alpha,\beta}$ were $1$-pseudoperiodic, it would be periodic and hence the letter $1$ would occur with rational density.
However, the $1$'s appear in ${\bf s}_{\alpha,\beta}$ with density
$\alpha$, which is irrational.

Now let $[0, c_1, c_2, \ldots]$
be the continued fraction expansion of $\alpha$.
Without loss of generality, we can assume that
$\alpha < 1/2$; otherwise consider ${\bf s}_{1-\alpha,0}$
which is the binary complement of ${\bf s}_{\alpha,0}$.   Hence
$c_1 \geq 2$.

It is easy to see from the definition of 
${\bf s}_{\alpha,\beta}$ that ${\bf s}_{\alpha,0}$ is a suffix of an infinite concatenation
of blocks of the form $0^{c_1-1}1$ and $0^{c_1} 1$.
It follows that $(c_1, c_1+1)$ is a pseudoperiod.
\end{proof}

\begin{remark}
There are, of course, non-Sturmian sequences that
are $2$-pseudoperiodic but not $1$-pseudoperiodic.  For example,
every sequence in $\{01,001\}^\omega$ has pseudoperiod $(2,3)$.
\end{remark}

\begin{remark}
Trivial observation:  to determine whether a given fixed
tuple $(p_1, p_2, \ldots, p_k)$ is a pseudoperiod
of an infinite sequence $\bf s$, it suffices
to examine all of the factors of length $p_k + 1$ of
$\bf s$.  
\end{remark}

\section{Shevelev's problems}
\label{sec3}

In this section, we consider some results of 
Vladimir \cite{Shevelev:2012}.  We reprove some of his results in a much simpler manner,
obtain new results, and completely solve one of his open questions.  

Recall from Section~\ref{notation} that the
Thue-Morse sequence ${\bf t} = 01101001\cdots $ is the infinite
fixed point, starting with $0$, of the map
sending $0 \rightarrow 01$ and $1 \rightarrow 10$.   Shevelev was interested in the pseudoperiodicity of $\bf t$,
and gave a number of theorems and open questions involving
this sequence.  We are able to prove all of the theorems and conjectures in 
\cite{Shevelev:2012} using our method, with the exception of his Conjecture 1.   Luckily, this conjecture was already proved by
Allouche \cite[Thm.~3.1]{Allouche:2015}.

\begin{proposition}
The Thue-Morse sequence is $3$-pseudoperiodic, but not $2$-pseudoperiodic.
\end{proposition}

\begin{proof}
The first statement follows from the (almost trivial) fact
that every word in $\{01, 10\}^\omega$ has pseudoperiod $(1,2,3)$.

For the second statement, we use \texttt{Walnut} again.
To prove the second half of the theorem, 
we assert $2$-pseudoperiodicity as follows and show that it is false:
\[
    \exists a,b \ (a \geq 1) \land (a < b) \land 
    \forall i\ ( {\bf t}[i] \in \{ {\bf t}[i+a], {\bf t}[i+b]\}).
\]
Translating the assertion into \texttt{Walnut}, we have:
\begin{verbatim}
eval twopseudotm "Ea,b (a>=1 & a<b) & Ai (T[i]=T[i+a] | T[i]=T[i+b])":
# returns FALSE
# 23 ms
\end{verbatim}
This returns \texttt{FALSE}, which proves that 
the Thue-Morse sequence is not $2$-pseudoperiodic.
\end{proof}

Since $\bf t$ is not $2$-pseudoperiodic, we know from 
Corollary~\ref{linear-bound} that there exists a constant $C$ such that
for $1 \leq a < b$ we have
${\bf t}[n] \not\in \{ {\bf t}[n+a], {\bf t}[n+b] \}$
for some $n \leq Cb$.   For the Thue-Morse
word, we can prove the following bound:
\begin{theorem}
\leavevmode
\begin{itemize}
\item[(a)] For all $a, b$ with $1 \leq a < b$ there exists
$n \leq {5\over3}b$ 
such that ${\bf t}[n] \not\in \{ {\bf t}[n+a], {\bf t}[n+b] \}$.

\item[(b)] The previous result is optimal, in the sense that if
the bound ${5\over 3}b$ is reduced, then there are infinitely many
counterexamples.
\end{itemize}
\label{tm-linear-bound}
\end{theorem}

\begin{proof}
To prove (a) and (b), we can use the following \texttt{Walnut} commands:
\begin{verbatim}
eval casea "Aa,b (1<=a & a<b) => En (3*n<=5*b) & 
   T[n]!=T[n+a] & T[n]!=T[n+b]":
# evaluates to TRUE

eval caseb "Am Ea,b 1<=a & a<b & b>m & Ai (3*i<5*b) => 
   (T[i]=T[i+a]|T[i]=T[i+b])":
# evaluates to TRUE
\end{verbatim}
\end{proof}

We now turn to Shevelev's Proposition 1 in \cite{Shevelev:2012}
which (in our terminology) asserts the following:
\begin{theorem}
The triples
$\{(a, a+2^k, a+2^{k+1})\, : \, a \geq 1, k \geq 0 \}$
are pseudoperiods for the Thue-Morse sequence.
\label{shev1}
\end{theorem}

Shevelev's proof of this was rather long and involved.
We can prove it almost instantly with \texttt{Walnut}, as follows:

\begin{proof}
We express the conditions placed on the triples as follows.
\begin{align*}
    \poweroftwo(x) &:= \exists k\ x=2^k \\
    \shevcond(a,b,c) &:= (a \geq 1) \land (\exists x\ \poweroftwo(x) \land (b=a+x) \land (c=a+2x)) .
\end{align*}
We write the proposition as:
\[ \forall a,b,c, i \ \shevcond(a,b,c) 
\implies ({\bf t}[i] \in \{ {\bf t}[i+a], {\bf t}[i+b], {\bf t}[i+c]\}). \]
Translating the above into \texttt{Walnut} commands,
we have:
\begin{verbatim}
reg power2 msd_2 "0*10*":
def shevcond "(a>=1) & (Ex $power2(x) & (b=a+x) & (c=a+2*x))":
# returns a DFA with 7 states
# 13 ms
eval prop1 "Aa,b,c,i $shevcond(a,b,c) => 
    (T[i]=T[i+a] | T[i]=T[i+b] | T[i]=T[i+c])":
# returns TRUE
# 6 ms
\end{verbatim}
The assertion returns \texttt{TRUE}, which proves that 
the Thue-Morse sequence is $3$-pseudoperiodic.
\end{proof}

Shevelev observed that Theorem~\ref{shev1} did not
characterize all such triples.  In his Proposition 2,
he showed $(1,8,9)$ is a pseudoperiod.  We can do this
with \texttt{Walnut} as follows:
\begin{verbatim}
eval shevprop2 "Ai (T[i]=T[i+1])|(T[i]=T[i+8])|(T[i]=T[i+9])":
# 97 ms
# return TRUE
\end{verbatim}

These two results caused Shevelev to pose his ``Open Question 1'',
which in our terminology is the following:
\begin{openproblem}
Characterize \emph{all} triples $(a,b,c)$ with $1 \leq a<b<c$ that are pseudoperiods
for the Thue-Morse sequence.
\end{openproblem}

Shevelev was unable to solve this, but using our 
methods, we can easily solve it.

\begin{theorem}
There is a DFA of $53$ states that accepts exactly the triples
$(a,b,c)$ such that $1 \leq a < b < c$ is a pseudoperiod of $\bf t$.
\end{theorem}

\begin{proof}
We want to characterize the triples 
$(a,b,c)$ such that 
\[\triple(a,b,c) :=
(a \geq 1) \land (a<b) \land (b<c) 
\land \forall i\ {\bf t}[i] \in \{ {\bf t}[i+a], {\bf t}[i+b], 
{\bf t}[i+c]\} .\]

We construct the following
DFA \texttt{triple} in \texttt{Walnut} 
to answer the question.
\begin{verbatim}
def triple "(a>=1) & (a<b) & (b<c) &
   Ai (T[i]=T[i+a] | T[i]=T[i+b] | T[i]=T[i+c])":
# returns a DFA with 53 states
# 4356513 ms
\end{verbatim} 
This gives us an automaton of 53 states, which is presented in the Appendix.
Determining it was a major calculation in \texttt{Walnut},
requiring 4356 seconds of CPU time and 18 GB of storage. 
The complete answer to Shevelev's question is then
the set of triples accepted by our DFA \texttt{triple}.
\end{proof}

Because the answer is so complicated, it is not that surprising
that Shevelev did not find a simple answer to his question.

Now that we have the automaton \texttt{triple},
we can easily check any triple $(a,b,c)$ to see if it
is a pseudoperiod of $\bf t$ in $O( \log abc)$ time,
merely by feeding the automaton with the base-$2$
representations of the triple $(a,b,c)$.

Furthermore, our automaton can be used to easily prove
other aspects of the pseudoperiods of the Thue-Morse sequence.   For example:
\begin{corollary}
\leavevmode
\begin{itemize}
\item[(a)]
For each $a \geq 1$ there exist arbitrarily large $b, c$ such that
$(a,b,c)$ is a pseudoperiod of $\bf t$.
\item[(b)] For each $b \geq 2$ there exist  pairs  $a,c$ 
such that
$(a,b,c)$ is a pseudoperiod of $\bf t$.
\item[(c)] For each $c \geq 3$ there exist pairs $a,b$ such that
$(a,b,c)$ is a pseudoperiod of $\bf t$.
\end{itemize}
\end{corollary}

\begin{proof}
We use the following \texttt{Walnut} code.
\begin{verbatim}
eval tmpa "Aa,m (a>=1) => Eb,c b>m & c>m & $triple(a,b,c)":
eval tmpb "Ab (b>=2) => Ea,c $triple(a,b,c)":
eval tmpc "Ac (c>=3) => Ea,b $triple(a,b,c)":
\end{verbatim}
and \texttt{Walnut} returns \texttt{TRUE} for all three.
\end{proof}

We now look at the possible distances between pseudoperiods of $\bf t$.
\begin{corollary}
\leavevmode
\begin{itemize}
\item[(a)] $\{ b-a \, : \, \exists c \ (a,b,c) 
\text{ is a pseudoperiod of } {\bf t} \} = 
\{ (2^j-1)2^i \, : \, j\geq 1, i \geq 0 \} \, \cup \,
\{ (2^{2j-1} + 1)2^i \, : \, j \geq 1, i \geq 0 \} \, \cup \, 
\{ 11 \cdot 2^i \, : \, i \geq 0 \} $.
\item[(b)] $\{ c-b \, : \, \exists a \ (a,b,c) 
\text{ is a pseudoperiod of } {\bf t} \} = 
\{ (2^j - 1)2^i \, : \, j \geq 1, i \geq 0 \} \, \cup \,
\{ (2^j + 1)2^i \, : \, j \geq 1, i \geq 0 \} $.
\end{itemize}
\end{corollary}

\begin{proof}
We use the following \texttt{Walnut} code.
\begin{verbatim}
reg parta msd_2 "0*11*0*|0*1(00)*10*|0*10110*":
reg partb msd_2 "0*100*10*|0*11*0*":
eval checka "An $parta(n) <=> (Ea,b,c $triple(a,b,c) & b=a+n)":
eval checkb "An $partb(n) <=> (Ea,b,c $triple(a,b,c) & c=b+n)":
\end{verbatim}
and \texttt{Walnut} returns \texttt{TRUE} twice.
\end{proof}

We now turn to Shevelev's Theorem 2 in \cite{Shevelev:2012}.
\begin{theorem}
    The only triples of distinct positive integers $(a,b,c)$ 
    for which both ${\bf t}[i]$ and
    $\overline{{\bf t}[i]}$ belong to
    $\{ {\bf t}[i+a], {\bf t}[i+b], {\bf t}[i+c] \}
    $
    for all $ i \geq 0 $ 
    are those satisfying
    $b=a+2^k$ and $c=a+2^{k+1}$
    for some $k \geq 0$.
\end{theorem}

\begin{proof}
To assert the claim in first-order logic, we first construct
a formula to show that at least one of the values in $S$ is
not equal to the other two;
this implies that $S$ contains both ${\bf t}[i]$ and $\overline{ {\bf t}[i] }$:
\begin{align*}
\neqtriple(a,b,c) & := (a \geq 1) \land (a<b) \land (a<c) \land  \\
&\forall i\ ({\bf t}[i+a] \neq {\bf t}[i+b] \lor {\bf t}[i+b] \neq 
{\bf t}[i+c] \lor {\bf t}[i+c] \neq {\bf t}[i+a]) .
\end{align*}
Our theorem can then be expressed as follows.
\[ \forall a,b,c \ (\triple(a,b,c) \land \neqtriple(a,b,c))
\iff \shevcond(a,b,c) .\]

Translating the above into \texttt{Walnut}, we build a DFA \texttt{neqtriple}.
\begin{verbatim}
def neqtriple "(a>=1) & (a<b) & (b<c) & 
    Ai (T[i+a]!=T[i+b] | T[i+b]!=T[i+c] | T[i+c]!=T[i+a])":
# returns a DFA with 7 states
# 554 ms
\end{verbatim}
We prove the theorem with the \texttt{Walnut} command below: 
\begin{verbatim}
eval thm2 "Aa,b,c ($triple(a,b,c) & $neqtriple(a,b,c)) <=> 
   $shevcond(a,b,c)":
# returns TRUE
# 6 ms
\end{verbatim}
This returns \texttt{TRUE}, which proves the Theorem.
\end{proof}

We now turn to Shevelev's Propositions 3 and 4 in \cite{Shevelev:2012}.
In our terminology, these are as follows:
\begin{proposition}
For all $k \geq 1$,
the Thue-Morse sequence has 
pseudoperiod $(a,b,c)$ 
if and only if 
it has pseudoperiod
$(2^k a, 2^k b, 2^k c)$.
\end{proposition}

\begin{proof}
We prove the following equivalent statement 
which implies the proposition by induction on $k$: 
\[ \forall a,b,c\ \triple(a,b,c) \iff \triple(2a, 2b, 2c). \]
Translating this into \texttt{Walnut}, we have the following. 
\begin{verbatim}
eval prop3n4 "Aa,b,c $triple(a,b,c) <=> $triple(2*a, 2*b, 2*c)":
# returns TRUE
# 14 ms
\end{verbatim}
This returns \texttt{TRUE}, which proves the proposition.
\end{proof}

\section{Other sequences}
\label{sec4}

After having obtained pseudoperiodicity results for the Thue-Morse
sequence {\bf t}, it is logical to try to obtain
similar results for other famous sequences.

In this section we examine sequences such as
the Rudin-Shapiro sequence {\bf rs},
the variant Thue-Morse sequence {\bf vtm},
the Tribonacci sequence {\bf tr}, and so forth.

For each sequence $s$ in this section, 
we assert $2$-pseudoperiodicity as follows and 
use \texttt{Walnut} to determine whether it holds:
\[
    \exists a,b \ (a \geq 1) \land (a < b) \land 
    \forall i\ (s_{i} \in \{s_{i+a}, s_{i+b}\}).
\]
And we assert $3$-pseudoperiodicity as follows and 
use \texttt{Walnut} to determine whether it holds:
\[
    \exists a,b,c \ (a \geq 1) \land (a < b) \land (b < c) \land 
    \forall i\ (s_{i} \in \{s_{i+a}, s_{i+b}, s_{i+c}\}) .
\]

\subsection{The Mephisto Waltz sequence}
The Mephisto Waltz sequence $\textbf{mw} = 001001110 \cdots$
is defined by the infinite fixed point of the morphism
$0 \to 001$, $1 \to 110$ starting with $0$.
It is sequence \seqnum{A064990} in the OEIS.
\begin{proposition}
    The Mephisto Waltz sequence is $3$-pseudoperiodic, but not $2$-pseudoperiodic.
\end{proposition}

\begin{proof}
We translate the assertions of $2$-pseudoperiodicity into \texttt{Walnut} as follows and show that it is false. 
\begin{verbatim}
eval twopseudomw "?msd_3 Ea,b (a>=1 & a<b) & 
    Ai (MW[i]=MW[i+a] | MW[i]=MW[i+b])":
# 496 ms
# return FALSE
\end{verbatim}
We translate the assertions of $3$-pseudoperiodicity into \texttt{Walnut} as follows and show that it is true. 
\begin{verbatim}
eval threepseudomw "?msd_3 Ea,b,c (a>=1 & a<b & b<c) & 
    Ai (MW[i]=MW[i+a] | MW[i]=MW[i+b] | MW[i]=MW[i+c])":
# 2202253 ms
# return TRUE
\end{verbatim}
\end{proof}

Knowing that the Mephisto Waltz sequence is $3$-pseudoperiodic naturally leads to the following problem.
\begin{problem}
Characterize \emph{all} triples $(a,b,c)$ with $1 \leq a<b<c$ that are pseudoperiods
for the Mephisto Waltz sequence.
\end{problem}
We want to characterize the triples 
$(a,b,c)$ such that 
\[\mwtriple(a,b,c) :=
(a \geq 1) \land (a<b) \land (b<c) 
\land \forall i\ {\bf mw}[i] \in \{ {\bf mw}[i+a], {\bf mw}[i+b],
{\bf mw}[i+c]\} .\]

We construct the following DFA \texttt{triplemw} in \texttt{Walnut} 
to solve the problem.
\begin{verbatim}
def triplemw "?msd_3 (a>=1 & a<b & b<c) & 
    Ai (MW[i]=MW[i+a] | MW[i]=MW[i+b] | MW[i]=MW[i+c])":
# returns a DFA with 13 states
# 2331762 ms
\end{verbatim} 
The complete answer to this problem is 
the set of triples accepted by our DFA \texttt{triplemw}.

\subsection{The ternary Thue-Morse sequence}
The ternary Thue-Morse sequence $\textbf{vtm} = 210201 \cdots$
is defined by the infinite fixed point of the morphism 
$2 \to 210$, $1 \to 20$, and $0 \to 1$ starting with 2. 
It is sequence \seqnum{A036577} in the OEIS.
\begin{proposition}
    The ternary (variant) Thue-Morse sequence is $3$-pseudoperiodic, 
    but not $2$-pseudoperiodic.
\end{proposition}

\begin{proof}
We translate the assertions of $2$-pseudoperiodicity into \texttt{Walnut} as follows and show that it is false. 
\begin{verbatim}
eval twopseudovtm "Ea,b (a>=1 & a<b) & 
   Ai (VTM[i]=VTM[i+a] | VTM[i]=VTM[i+b])":
# 235 ms
# return FALSE
\end{verbatim}
We translate the assertions of $3$-pseudoperiodicity into \texttt{Walnut} as follows and show that it is true. 
\begin{verbatim}
eval threepseudovtm "Ea,b,c (a>=1 & a<b & b<c) & 
    Ai (VTM[i]=VTM[i+a] | VTM[i]=VTM[i+b] | VTM[i]=VTM[i+c])":
# 505315560 ms
# 188 GB
# return TRUE
\end{verbatim}
\end{proof}

Knowing that the ternary Thue-Morse sequence is $3$-pseudoperiodic naturally leads to the following problem.
\begin{problem}
Characterize \emph{all} triples $(a,b,c)$ with $1 \leq a<b<c$ that are pseudoperiods
for the ternary Thue-Morse sequence.
\end{problem}
We want to characterize the triples $(a,b,c)$ such that 
\[\vtmtriple(a,b,c) :=
(a \geq 1) \land (a<b) \land (b<c) 
\land \forall i\ ( {\bf vtm}[i] \in \{
{\bf vtm}[i+a], {\bf vtm}[i+b], {\bf vtm}[i+c]\}) .\]

We construct the following DFA \texttt{triplevtm} in \texttt{Walnut} 
to solve the problem.
\begin{verbatim}
def triplevtm "(a>=1 & a<b & b<c) & 
    Ai (VTM[i]=VTM[i+a] | VTM[i]=VTM[i+b] | VTM[i]=VTM[i+c])":
# returns a DFA with 12 states
# 815830898 ms
\end{verbatim} 
The complete answer to this problem is  the set of triples accepted
by our DFA \texttt{triplevtm}.  It is depicted in Figure~\ref{fig6}.
Again, this was a very large computation with \texttt{Walnut}.
\begin{figure}[H]
\begin{center}
    \includegraphics[width=5.5in]{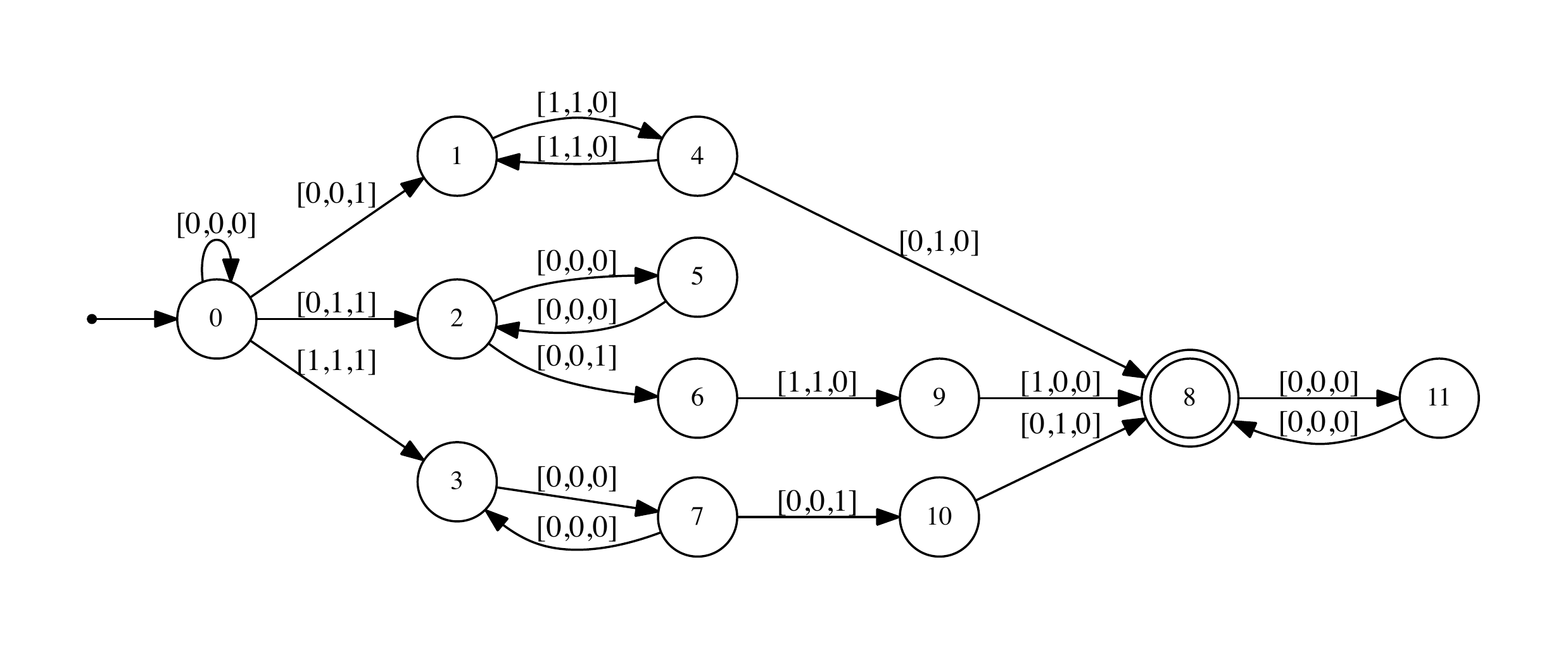}
\end{center}
\caption{Automaton recognizing all pseudoperiods of size $3$ for the {\bf vtm} sequence.}
\label{fig6}
\end{figure}

By looking at the acceptance paths of Figure~\ref{fig6}, we can deduce the following result.
\begin{theorem}
The only $3$-pseudoperiods for the {\bf vtm} sequence are
\begin{itemize}
    \item $\{ ( (2^{2i+1} - 1)2^{2j+1}, (2^{2i+2} - 1) 2^{2j}, 2^{2i+2j+2}) \suchthat i, j \geq 0 \}$
    \item $\{ (3\cdot 2^{2j}, (2^{2i+2} + 1)2^{2j+1}, (2^{2i+1} + 1)2^{2j+2} ) \suchthat i,j \geq 0 \}$
    \item $\{ (2^{2i+2j+3}, (2^{2i+3} + 1) 2^{2j},
    (2^{2i+2}+1)2^{2j+1}) \suchthat i, j \geq 0 \}$.
\end{itemize}
\end{theorem}

\begin{proof}
There are only essentially three possible paths to the accepting state labeled 8 in Figure~\ref{fig6}.  They are labeled 
\begin{itemize}
\item $[0,0,0]^*[0,0,1][1,1,0]([1,1,0][1,1,0])^* [0,1,0]([0,0,0][0,0,0])^*$
\item $[0,0,0]^* [0,1,1]([0,0,0][0,0,0])^* [0,0,1][1,1,0][1,0,0] ([0,0,0][0,0,0])^*$
\item $[0,0,0]^* [1,1,1] [0,0,0] ([0,0,0][0,0,0])^* [0,0,1][0,1,0] ([0,0,0][0,0,0])^*$
\end{itemize}
By considering the base-$2$ numbers specified by each coordinate, we obtain the theorem.
\end{proof}

\subsection{The period-doubling sequence}
The period-doubling sequence $\textbf{pd} = 1011101011 \cdots$ is defined
by the infinite fixed point of the morphism $1 \to 10$, $0 \to 11$ starting with 1.
It is sequence \seqnum{A035263} in the OEIS. 
\begin{proposition}
    The period-doubling sequence is $3$-pseudoperiodic, but not $2$-pseudoperiodic.
\end{proposition}

\begin{proof}
We translate the assertions of $2$-pseudoperiodicity into \texttt{Walnut} as follows and show that it is false. 
\begin{verbatim}
eval twopseudopd "Ea,b (a>=1 & a<b) & 
    Ai (PD[i]=PD[i+a] | PD[i]=PD[i+b])":
# 424 ms
# return FALSE
\end{verbatim}
We translate the assertions of $3$-pseudoperiodicity into \texttt{Walnut} as follows and show that it is true. 
\begin{verbatim}
eval threepseudopd "Ea,b,c (a>=1 & a<b & b<c) & 
    Ai (PD[i]=PD[i+a] | PD[i]=PD[i+b] | PD[i]=PD[i+c])":
# 40 ms
# return TRUE
\end{verbatim}
\end{proof}

Knowing that the period-doubling sequence is $3$-pseudoperiodic naturally leads to the following problem. 
\begin{problem}
Characterize \emph{all} triples $(a,b,c)$ with $1 \leq a<b<c$ that are pseudoperiods for the period-doubling sequence.
\end{problem}
We want to characterize the triples 
$(a,b,c)$ such that 
\[\pdtriple(a,b,c) :=
(a \geq 1) \land (a<b) \land (b<c) 
\land \forall i\ ( {\bf pd}[i] \in \{ {\bf pd}[i+a], 
{\bf pd}[i+b], {\bf pd}[i+c]\}) .\]

We construct the following
DFA \texttt{triplepd} in \texttt{Walnut} 
to solve the problem.
\begin{verbatim}
def triplepd "(a>=1 & a<b & b<c) & 
    Ai (PD[i]=PD[i+a] | PD[i]=PD[i+b] | PD[i]=PD[i+c])":
# returns a DFA with 28 states
# 30 ms
\end{verbatim} 
The complete answer to this problem is 
the set of triples accepted by our DFA \texttt{triplepd}.

\subsection{The Rudin-Shapiro sequence}
The Rudin-Shapiro sequence $\textbf{r} = 00010010 \cdots$
is defined by the relation $\textbf{r}[n] = |(n)_2|_{11}$ mod 2, that is,
the number of occurrences of $11$, computed modulo $2$,
in the base-2 representation of $n$.
It is sequence \seqnum{A020987} in the OEIS. 

\begin{theorem}
The Rudin-Shapiro sequence is $4$-pseudoperiodic, but
not $3$-pseudoperiodic.
\end{theorem}

\begin{proof}
To check $3$-pseudoperiodicity, we used the \texttt{Walnut} command
\begin{verbatim}
eval rudinpseudo "Ea,b,c a>=1 & a<b & b<c & 
   An (RS[n]=RS[n+a]|RS[n]=RS[n+b]|RS[n]=RS[n+c])":
\end{verbatim}
which returned the result \texttt{FALSE}.  This was
a big computation, requiring 20003988ms and more
than 200 GB of memory on a 64-bit machine.

 It is $4$-pseudoperiodic, as \texttt{Walnut} can easily verify
   that $(2,3,4,5)$ is a pseudoperiod.
\end{proof}

\subsection{The Tribonacci sequence}
The Tribonacci sequence 
is a generalization of the Fibonacci sequence.
It is defined by the infinite fixed point of the morphism 
$0 \to 01$, $1 \to 02$, and $2 \to 0$
and is sequence \seqnum{A080843} in the OEIS.

\begin{theorem}
The Tribonacci sequence is $3$-pseudoperiodic, but
not $2$-pseudoperiodic.
\end{theorem}

\begin{proof}
It has pseudoperiod $(4,6,7)$, as can be easily verified by checking all factors of length $8$ (or with \texttt{Walnut}).
\end{proof}

\begin{openproblem}
Characterize all the $3$-pseudoperiods of the
Tribonacci sequence.
\label{open1}
\end{openproblem}

Although this is in principle doable with \texttt{Walnut}, 
so far, this seems to be beyond our computational abilities, requiring the
determinization of a large nondeterministic automaton.

\subsection{The paperfolding sequences}

The paperfolding sequences are an uncountable family of sequences originally 
introduced by \cite{Davis&Knuth:1970} and later studied 
by \cite{Dekking&MendesFrance&vanderPoorten:1982}.   The first-order theory of the paperfolding sequences was proved decidable in
\cite{Goc&Mousavi&Schaeffer&Shallit:2015}.   Every infinite paperfolding sequence is specified by an infinite sequence $\bf f$ of unfolding instructions.   Since \texttt{Walnut}'s automata work on finite strings---they are not B\"uchi automata---we have to approximate an infinite $\bf f$
by considering its finite prefixes $f$.  A fuller discussion of exactly how to do this can be
found in \cite[Chap.~12]{Shallit:2022}; we just sketch
the ideas here.

We can use \texttt{Walnut} to determine the pseudoperiods of any specific paperfolding sequence, or the pseudoperiod common to all paperfolding sequences.  

\texttt{Walnut} can prove that no paperfolding sequence is $2$-pseudoperiodic, as follows:
\begin{verbatim}
reg linkf {-1,0,1} {0,1} "()*[0,1][0,0]*":
def pffactoreq "?lsd_2 At (t<n) => FOLD[f][i+t]=FOLD[f][j+t]":
eval paper_pseudo2 "?lsd_2 Ef,a,b,x 1<=a & a<b & $linkf(f,x) & 
   x>=2*b+3 & Ai (i>=1 & i+b+1<=x) => 
   ($pffactoreq(f,i,i+a,1)|$pffactoreq(f,i,i+b,1))":
# FALSE, 26926 secs
\end{verbatim}
Here \texttt{pffactoreq} asserts that the two length-$n$
factors of the paperfolding sequence specified by a finite code
$f$, one beginning at position $i$ and one at position $j$ are the same.  And 
\texttt{linkf} asserts
that $x = 2^{|f|}$.   The assertion \texttt{paper\_pseudo2} is that there exists some paperfolding sequence and numbers $a,b$ such that every position $i$ has a symbol
equal to either the symbol at position $i+a$ or
$i+b$.   

All paperfolding sequences are $3$-pseudoperiodic; 
for example, $(1,3,4)$ is a pseudoperiod of all paperfolding sequences.
 \begin{verbatim}
eval paper_pseudo134 "?lsd_2 Af,x,i ($linkf(f,x) & i>=1 & i+5<=x) => 
   ($pffactoreq(f,i,i+1,1)|$pffactoreq(f,i,i+3,1)|
   $pffactoreq(f,i,i+4,1))":
\end{verbatim}
However, not all pseudoperiods work for all paperfolding sequences.
For example, we can use \texttt{Walnut} to show that $(1,2,16)$ is a pseudoperiod for the paperfolding sequence specified by the unfolding instructions $\overline{1} \, 1 \, 1 \, \cdots$, but
not a pseudoperiod for the regular paperfolding sequence (specified by $1 \, 1 \, 1 \, \cdots$).   

We can compute the pseudoperiods that work for all paperfolding sequences simultaneously, using
the following \texttt{Walnut} code:
\begin{verbatim}
def paper_pseudo3 "?lsd_2 1<=a & a<b & b<c & 
   Af,x,i ($linkf(f,x) & i>=1 & i+c+1<=x) => 
   ($pffactoreq(f,i,i+a,1)|$pffactoreq(f,i,i+b,1)|$pffactoreq(f,i,i+c,1))":
# 10 states, 2356 ms
\end{verbatim}

The automaton in Figure~\ref{paperfig}
accepts the base-$2$ representation (here, \emph{least significant digit first}) of those triples $(a,b,c)$ with
$1 \leq a < b < c$ as a pseudoperiod for all paperfolding sequences.
\begin{figure}[H]
\begin{center}
    \includegraphics[width=5.5in]{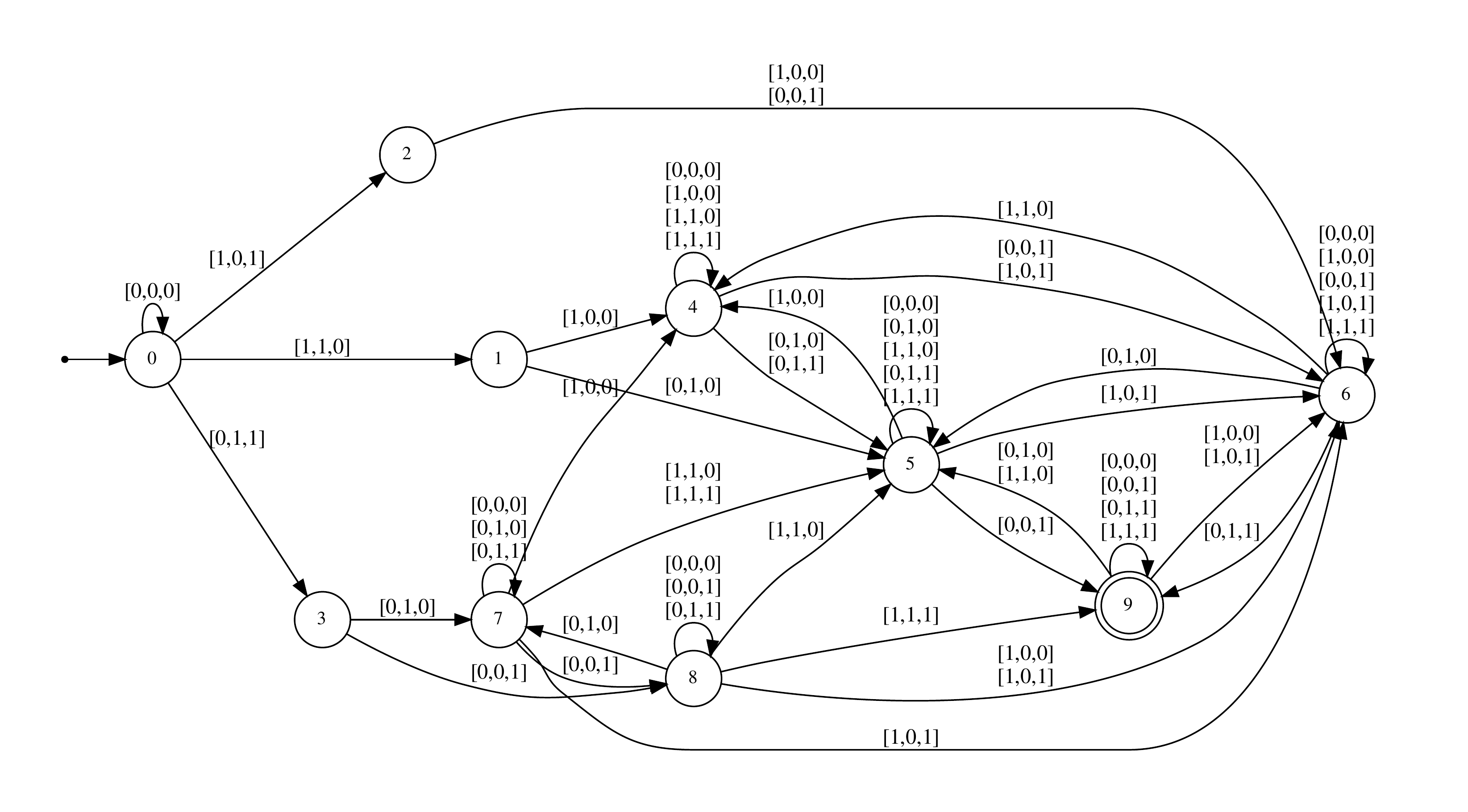}
\end{center}
\caption{Automaton accepting base-$2$ representations of pseudoperiod triples common to all paperfolding sequences.}
\label{paperfig}
\end{figure}

\section{Critical exponents}
\label{sec5}

In this section we consider the following problem.
Suppose we consider the class $C_{a,b}$ of all infinite binary words
with a specified pseudoperiod $(a,b)$, for integers $1 \leq a < b$. 
Can we construct words of small critical exponent in $C_{a,b}$?
And what is the repetition threshold of $C_{a,b}$?

The general strategy we employ is the following.
We use a heuristic search procedure to try to guess a morphism $h_{a,b}$ such that
either $h_{a,b}({\bf t})$ or $h_{a,b}({\bf vtm})$ has pseudoperiod $(a,b)$
and avoids $e^+$ powers for some suitable exponent $e$.
Once such an $h_{a,b}$ is found, we can verify its correctness using 
\texttt{Walnut}.   
Simultaneously we can do a breadth-first search over the tree of
all binary words having pseudoperiod $(a,b)$ and avoiding $e$-powers.
If this tree turns out to be finite, we have proved the optimality of this $e$.

Our first result shows that this critical exponent can never be $\leq 7/3$.

\begin{theorem}
If $\bf x$ is an infinite binary word that is $2$-pseudoperiodic,
then $x$ contains a $(7/3)$-power.   
\end{theorem}

\begin{proof}
Suppose $\bf x$ has pseudoperiod $1 \leq a < b$, but
is $(7/3)$-power-free.
Theorem 6 of 
\cite{Karhumaki&Shallit:2004} says that every infinite
$(7/3)$-power-free binary word 
contains factors of the
form $\mu^i (0)$ for all $i \geq 0$.  These factors are all prefixes
of $\bf t$.

However, as we have seen in Theorem~\ref{tm-linear-bound}, the prefix
of length ${5\over 3}b + 1$ of $\bf t$ cannot have pseudoperiod
$a,b$, as the relation
${\bf t}[n] \in \{ {\bf t}[n+a], {\bf t}[n+b] \} $ is violated
for some $n\leq {5\over 3}b$.  Thus it suffices to choose
$i$ large enough such that $2^i \geq {5 \over 3} b + 1$.
This contradiction proves the result.
\end{proof}

Next we consider the case $b = 2a$.
\begin{proposition}
Let $a\geq 1$ be an integer.
If an infinite word has pseudoperiod $(a,2a)$,
then it has critical exponent $\infty$.
\end{proposition}

\begin{proof}
Suppose that $\bf x$ has pseudoperiod $(a,2a)$.  From $\bf x$ extract
the subsequences $${\bf x}_{a,i} = ({\bf x}[an+i])_{n \geq 0}$$ 
corresponding to indices that are congruent to $i$ (mod $a$), for
$0 \leq i < a$.  Clearly each such subsequence has pseudoperiod $(1,2)$.
By Proposition~\ref{pp12}, each subsequence ${\bf x}_{a,i}$ must be of the form
$c^\omega$ or $c^* (cd)^\omega$ for $c, d$ distinct letters. It now follows that $\bf x$ is eventually periodic
with period $2a$, and hence has infinite critical exponent.
\end{proof}

\begin{proposition}
Let $\alpha\geq 2^+$.
If an $\alpha$-free (resp., $\alpha^+$-free) binary word $w$ has pseudoperiod $(a_1, \ldots, a_k)$, then $\mu(w)$ is an $\alpha$-free (resp., $\alpha^+$-free) binary word with pseudoperiod $(2a_1, \ldots, 2a_k)$.
\label{doublepp}
\end{proposition}

\begin{proof}
The claim about the pseudoperiods is clear. The result about power-freeness
can be found in, e.g., \cite[Theorem 5]{Karhumaki&Shallit:2004}.
\end{proof}

We now summarize our results on critical exponents in the following table.
Each entry corresponding to a pseudoperiod $(a,b)$ with $b \not=2a$ has three entries:
\begin{itemize}
\item[(a)] upper left:  an exponent $e$, where the repetition threshold for $C_{a,b}$ is $e^+$;
\item[(b)] upper right:  the length of the longest finite word having pseudoperiod $(a,b)$ and avoiding $e$-powers;
\item[(c)] lower line: the morphic word with pseudoperiod $(a,b)$ and avoiding $e^+$ powers.
\end{itemize}

\begin{table}[htb]
\begin{center}
\resizebox{.95\textwidth}{!}{
\begin{tabular}{|c|c|c|c|c|c|c|c|c|c|c|c|c|}
\hline
\backslashbox{$a$}{$b$} & 2 & 3 & 4 & 5 & 6 & 7 & 8 & 9 & 10 & 11 & 12 \\
\hline
1 & $\infty$ & $5/2 \quad 33$     & $3 \quad 11$       & $13/5 \quad 29$    & $7/3 \quad 15$       & $3 \quad 61$       & $3 \quad 45$       & $5/2 \quad 43$     & $5/2 \quad 33$      & $5/2 \quad 52$      & $5/2 \quad 57$ \\
  &          & $h_{1,3}({\bf t})$ & $h_{1,4}({\bf t})$ & $h_{1,5}({\bf t})$ & $h_{1,6}({\bf vtm})$ & $h_{1,7}({\bf t})$ & $h_{1,8}({\bf t})$ & $h_{1,9}({\bf t})$ & $h_{1,10}({\bf t})$ & $h_{1,11}({\bf t})$ & $h_{1,12}({\bf t})$ \\
\hline
2 &   & $13/5 \quad 30$    & $\infty$ & $3 \quad 15$       & $5/2 \quad 66$          & $13/5 \quad 84$    & $13/5 \quad 30$    & $5/2 \quad 19$     & $13/5 \quad 60$     & $5/2 \quad 20$      & $7/3 \quad 31$ \\
  &   & $h_{1,5}({\bf t})$ &          & $h_{2,5}({\bf t})$ & $\mu(h_{1,3}({\bf t}))$ & $h_{2,7}({\bf t})$ & $h_{1,5}({\bf t})$ & $h_{2,9}({\bf t})$ & $\mu(h_{1,5}({\bf t}))$ & $h_{2,11}({\bf t})$ & $\mu(h_{1,6}({\bf vtm}))$ \\
\hline
3 &   &   & $5/2 \quad 33$     & $13/5 \quad 34$    & $\infty$ & $13/5 \quad 98$    & $5/2 \quad 42$     & $8/3 \quad 28$     & $13/5 \quad 69$    & $5/2 \quad 59$     & $8/3 \quad 72$ \\
  &   &   & $h_{1,3}({\bf t})$ & $h_{1,5}({\bf t})$ &          & $h_{1,5}({\bf t})$ & $h_{1,3}({\bf t})$ & $h_{3,9}({\bf t})$ & $h_{1,5}({\bf t})$ & $h_{1,3}({\bf t})$ & $h_{3,12}({\bf t})$ \\
\hline
4 &   &   &   & $3 \quad 21$       & $7/3 \quad 40$     & $3 \quad 61$       & $\infty$ & $7/3 \quad 18$       & $5/2 \quad 33$      & $5/2 \quad 19$      & $5/2 \quad 141$ \\
  &   &   &   & $h_{4,5}({\bf t})$ & $h_{4,6}({\bf t})$ & $h_{1,7}({\bf t})$ &          & $h_{1,6}({\bf vtm})$ & $h_{1,10}({\bf t})$ & $h_{4,11}({\bf t})$ & $\mu^2(h_{1,3}({\bf t}))$ \\
\hline
5 &   &   &   &   & $5/2 \quad 66$     & $3 \quad 68$       & $13/5 \quad 33$    & $5/2 \quad 66$     & $\infty$ & $5/2 \quad 20$      & $18/7 \quad 158$ \\
  &   &   &   &   & $h_{5,6}({\bf t})$ & $h_{2,5}({\bf t})$ & $h_{1,5}({\bf t})$ & $h_{2,6}({\bf t})$ &          & $h_{2,11}({\bf t})$ & $h_{5,12}({\bf t})$ \\
\hline
6 &   &   &   &   &   & $7/3 \quad 40$     & $5/2 \quad 60$     & $17/6 \quad 89$    & $7/3 \quad 48$      & $5/2 \quad 69$      & $\infty$ \\
  &   &   &   &   &   & $h_{4,6}({\bf t})$ & $\mu(h_{1,3}({\bf t}))$ & $h_{6,9}({\bf t})$ & $h_{6,10}({\bf t})$ & $h_{1,11}({\bf t})$ & \\
\hline
7 &   &   &   &   &   &   & $13/5 \quad 50$    & $7/3 \quad 41$     & $13/5 \quad 92$    & $13/5 \quad 84$ & $7/3 \quad 31$ \\
  &   &   &   &   &   &   & $h_{1,5}({\bf t})$ & $h_{4,6}({\bf t})$ & $h_{2,7}({\bf t})$ & $h_{7,11}({\bf t})$ & $h_{1,6}({\bf vtm})$ \\
\hline
8 &   &   &   &   &   &   &   & $5/2 \quad 66$     & $5/2 \quad 33$      & $3 \quad 65$        & $7/3 \quad 82$ \\
  &   &   &   &   &   &   &   & $h_{8,9}({\bf t})$ & $h_{1,10}({\bf t})$ & $h_{8,11}({\bf t})$ & $\mu(h_{4,6}({\bf t}))$ \\
\hline
9 &   &   &   &   &   &   &   &   & $7/3 \quad 40$      & $5/2 \quad 57$      & $55/21 \quad 200$ \\
  &   &   &   &   &   &   &   &   & $h_{6,10}({\bf t})$ & $h_{9,11}({\bf t})$ & $h_{9,12}({\bf t})$  \\
\hline
10 &   &   &   &   &   &   &   &   &   & $5/2 \quad 33$      & $7/3 \quad 54$\\
   &   &   &   &   &   &   &   &   &   & $h_{1,10}({\bf t})$ & $h_{10,12}({\bf t})$\\
\hline
11 &   &   &   &   &   &   &   &   &   &   & $7/3 \quad 31$ \\
   &   &   &   &   &   &   &   &   &   &   & $h_{11,12}({\bf vtm})$\\
\hline
\end{tabular}}
\end{center}
\caption{Optimal critical exponents for binary words with certain specified pseudoperiod.}
\label{tab1}
\end{table}

See the files 
\texttt{longest\_finite\_seqs.txt} and
\texttt{critical\_exp\_morphisms.txt} at\\
\centerline{
\url{https://github.com/sonjashan/sha_gen.git}}\\
for the specific morphisms.

From examination of Table~\ref{tab1}, we see that all the critical
exponents are at most $3^+$. This leads to the following conjecture.

\begin{conjecture}
For all pairs $(a,b)$ with $1 \leq a < b$ and $b \not= 2a$,
there exists an infinite binary word with pseudoperiod $(a,b)$
and avoiding $3^+$-powers.
\label{conj1}
\end{conjecture}

We verified this conjecture for $1 \leq a < b \leq 54$.
For each pair of $a$ and $b$, we first try each previously
saved morphism $h$ on the Thue-Morse sequence $\bf t$
to see if $h(\bf t)$ has pseudoperiod $\{a,b\}$ and avoids $3^+$-powers.
If that fails, we use backtracking to search for a new morphism that
meets the criteria. 
Once we find such an morphism, we verify the pseudoperiodicity
and the powerfreeness with Walnut and save the morphism for future use.

The following morphism is an example. 
It is initially generated for $a=1$ and $b=5$
but it also works for 122 other pairs of $a$ and $b$ we tested. 

\begin{verbatim}
morphism sha3 "0->11100011000 1->11100111000":
image S3 sha3 T:
eval pp1_5_checkS3 "An (S3[n]=S3[n+1]|S3[n]=S3[n+5])":
eval cubeplusfree_S3 "~Ei,n n>0 & Aj (j<=2*n) => S3[i+j] = S3[i+j+n]":
\end{verbatim}

For more details on this implementation, please see the github repository at \\
\centerline{
\url{https://github.com/sonjashan/sha_gen.git} \ .}

Let us also provide details about the exceptional case of $a=1$ and $b=6$.
The morphic word with pseudoperiod $(1,6)$ which avoids $(7/3)^+$-powers
is $h_{1,6}({\bf vtm})$, where
$h_{1,6}(0) = 0011011001011001$,
$h_{1,6}(1) = 0011011001$, and
$h_{1,6}(2) = 001101$.

A simple computation shows that $h({\bf vtm})$ has pseudoperiod $(1,6)$
and that its factors of length $1000$ avoid $(7/3)^+$-powers.

Suppose that $h_{1,6}({\bf vtm})$ contains a factor $w$ that is a $(7/3)^+$-power. Thus $|w| > 1000$.
Notice that the factor $0011$ is a common prefix of the $h_{1,6}$-image of all three letters.
Moreover, $0011$ appears in $h_{1,6}({\bf vtm})$ only as the prefix of the $h$-image of a letter.

We consider the word $w'$ obtained from $w$ by erasing the smallest prefix of $w$ such that $w'$ starts with $0011$.
Since we erase at most $|h(0)|-1=15$ letters, the word $w'$ is a repetition of 
period $p$ and exponent at least $2.2$.

So $w'[1..4]=w'[p+1..p+4]=w'[2p+1..2p+4]=0011$.
This implies that $w'[1..2p]=h(uu)$ where the pre-image $uu$ must be a factor of $\bf vtm$.
This is a contradiction, since $\bf vtm$ is squarefree.

Finally, the results with a morphic word using $\mu$ as outer morphism
are obtained via Proposition~\ref{doublepp}.

\subsection{Binary words with pseudoperiods of the form $(1,a)$}

\begin{theorem}
For at least 85\% of all positive integers
$a \geq 3$ there is an infinite binary word with pseudoperiod $(1,a)$,
and avoiding $3^+$-powers.
\label{percent}
\end{theorem}

\begin{proof}
The idea is to search for words with the given properties that have pseudoperiod
$(1,a)$ for all $a$ in a given residue class $a \equiv \modd{i} {n}$.
As before, our words are constructed by applying an $n$-uniform morphism (obtained by a heuristic search)
to the Thue-Morse word $\bf t$, and then correctness is verified with 
\texttt{Walnut}.

Our results are summarized in Table~\ref{tab2}.
\begin{table}[htb]
\begin{minipage}{.45\textwidth}
\begin{center}
\begin{tabular}{|c|c|l|}
\hline
$i$ & $n$ & morphism \\
\hline
4 & 5 & $0 \rightarrow 00011$ \\
  &   & $1 \rightarrow 00111$ \\
  \hline
3 & 7 & $0 \rightarrow 0010011$ \\
  &   & $1 \rightarrow 0011011$ \\
  \hline
4 & 9 & $0 \rightarrow 000100011$ \\
  &   & $1 \rightarrow 000110011$ \\
  \hline
8 & 9 & $0 \rightarrow 110011000$ \\
  &   & $1 \rightarrow 110011100$ \\
  \hline
4 & 11 & $0 \rightarrow  11000111000$\\
  &    & $1 \rightarrow  11000111001$\\
  \hline
5 & 11 & $0 \rightarrow  11000111000$\\
  &    & $1 \rightarrow  11000111001$\\
  \hline
7 & 11 & $0 \rightarrow  10011001000$\\
  &    & $1 \rightarrow  10011001001$\\
  \hline
8 & 11 & $0 \rightarrow 10110100100$\\
  &    & $1 \rightarrow 10110100101$\\
  \hline
10 & 11 & $0 \rightarrow 11000111000$\\
   &    & $1 \rightarrow 11000111001$\\
  \hline
\end{tabular}
\end{center} 
\end{minipage}
\quad\quad
\begin{minipage}{.45\textwidth}
\begin{center}
\begin{tabular}{|c|c|l|}
\hline
$i$ & $n$ & morphism \\
\hline
5 & 13 & $0 \rightarrow  1010010110100$\\
  &    & $1 \rightarrow  1010010110101$\\
  \hline
8 & 13 & $0 \rightarrow  1100110001000$\\
  &    & $1 \rightarrow  1100110001001$\\
  \hline
4 & 14 & $0 \rightarrow 11000100011000$ \\
  &    & $1 \rightarrow 11000100011001$ \\
  \hline
9 & 14 & $0 \rightarrow 11001110011000$ \\
  &    & $1 \rightarrow 11001110011101$ \\
  \hline
13 & 14 & $0 \rightarrow 11001100111000$ \\
  &     & $1 \rightarrow 11001100011001$ \\
  \hline
7 & 15 & $0 \rightarrow 110110011001000$ \\
  &    & $1 \rightarrow 110110011001001$ \\
  \hline
4 & 16 & $0 \rightarrow 1000111000111000$ \\
  &    & $1 \rightarrow 1000111000111001$ \\
  \hline
6 & 16 & $0 \rightarrow 1011001001101000$ \\
  &    & $1 \rightarrow 1011001001101001$ \\
\hline
10 & 16 & $0 \rightarrow 1000111000111000$ \\
   &    & $1 \rightarrow 1000111000111001$ \\
  \hline
15 & 16 & $0 \rightarrow 1100111000111000$ \\
   &    & $1 \rightarrow 1100111000110001$ \\
\hline
\end{tabular}
\end{center}
\end{minipage}
\caption{Words avoiding $3^+$ powers with pseudoperiods in residue classes.}
\label{tab2}
\end{table}

As an example, here is the \texttt{Walnut} code
verifying the results for $(i,n) = (4,5)$:
\begin{verbatim}
morphism a45 "0->00011 1->00111":
image B45 a45 T:
eval cube45 "~Ei,n (n>=1) & At (t<=2*n) => B45[i+t]=B45[i+n+t]":
eval test45 "Ap (Ek p=5*k+4) => An (B45[n]=B45[n+1]|B45[n]=B45[n+p])":
\end{verbatim}
and both commands return \texttt{TRUE}.

The residue classes in Table~\ref{tab2}
correspond to $n = 5,7,9,11,13,14,15,16$.
Now \\ $\lcm(5,7,9,11,13,14,15,16) = 720720$, and the residue classes above cover
$614614$ of the possible residues (mod $720720$).
So we have covered $614614/720720 \doteq .852$ of all the possible~$a$.
\end{proof}

Theorem~\ref{percent} can obviously be improved by considering larger moduli.
For example, there exists a morphism for every residue class modulo $41$
except $0, 1, 2, 5, 6, 21, 23, 39$.

\section{Larger alphabets}
\label{sec6}

Up to now we have been mostly concerned with binary words.
In this section we consider pseudoperiodicity in larger alphabets.

The (unrestricted) repetition threshold $RT(k)$ for words over $k$ letters
is well-known: we have $RT(3)=7/4$, $RT(4)=7/5$, and $RT(k)=k/(k-1)$
if $k=2$ or $k\geq 5$ \cite{Currie&Rampersad:2011,Rao:2011}.
Notice that the words attaining the repetition threshold
are necessarily $3$-pseudoperiodic. Indeed, every infinite $(k-1)/(k-2)$-free
word over $k\geq 3$ letters is $(k-1,k,k+1)$-periodic.
Thus, it remains to investigate $2$-pseudoperiodic words.

Let us consider the repetition threshold $RT'(k)$
for $2$-pseudoperiodic words over $k$ letters.
Obviously, $RT(k)\leq RT'(k)$. From the previous section, we know that
$RT'(2)=7/3$. The following results show that $RT'(3)=7/4$, $RT'(4)\leq 3/2$,
and $RT'(5)\leq 4/3$, respectively.

\begin{theorem}
The image of every $(7/5)^+$-free word over 4 letters by the following 188-uniform
morphism avoids $(7/4)^+$-powers and has pseudoperiod $(18, 37)$.

\begin{align*}
    0 \rightarrow\ & p201021201210120102120210201202120121020102120121012010210\\
                   & 1210201202120121020102101201020120210121021201210120102120\\
    1 \rightarrow\ & p201021201210120102101210201202120121021202101201020120210\\
                   & 1210212012101201021202101201020120210201021201210120102120\\
    2 \rightarrow\ & p201021201210120102101210201202120121020102101201020120210\\
                   & 1210212012101201021202102012021201210201021201210120102101\\
    3 \rightarrow\ & p121021201210120102120210120102012021020102120121012010210\\
                   & 1210201202120121021202101201020120210121021201210120102120\\
\end{align*}
where $p = 2102012021201210201021012010201202101210201202120121021202101201020120210$.




\label{dejean3}
\end{theorem}

\begin{theorem}
The image of every $(7/5)^+$-free word over 4 letters by the following 170-uniform
morphism avoids $(3/2)^+$-powers and has pseudoperiod $(4, 10)$.

\begin{align*}
    0 \rightarrow\ & p301020323132102010313231201020323130102012313230201021323120102032\\
                   & 313210201031323020102132313010203231321020123132302010313231201020\\
    1 \rightarrow\ & p201020323132102012313230201021323130102012313210201031323120102132\\
                   & 313010203231321020103132302010213231301020123132302010313231201020\\
    2 \rightarrow\ & p201020323132102010313231201021323130102032313210201231323020102132\\
                   & 313010201231321020103132312010203231301020123132302010313231201021\\
    3 \rightarrow\ & p201020323132102010313231201021323130102012313230201021323120102032\\
                   & 313010201231321020103132312010203231321020123132302010313231201021\\
\end{align*}
where $p = 32313010201231321020103132302010213231$.



\label{dejean4}
\end{theorem}

\begin{theorem}
The image of every $(5/4)^+$-free word over 5 letters by the following 84-uniform
morphism avoids $(4/3)^+$-powers and has pseudoperiod $(9, 19)$.

\begin{align*}
    0 \rightarrow\ & p312402104302403104201403204230243210230140210420124320423124021423024031\\
    1 \rightarrow\ & p312402104301403204231203210230140210430120310420124320423124021423024032\\
    2 \rightarrow\ & p012432102312402104301403104201243204230140210430120310423024321023014031\\
    3 \rightarrow\ & p012402104302403104201243210231240210420140320423120321423024321023014032\\
    4 \rightarrow\ & p012402104301403104201243214231240210420140320423124321423024031023014032\\
\end{align*}
where $p = 043012032142$.



\label{dejean5}
\end{theorem}

Theorems~\ref{dejean3},~\ref{dejean4}, and~\ref{dejean5} make use of~\cite[Lemma~2.1]{Ochem:2006},
which has been recently extended to larger exponents in~\cite[Lemma~23]{Mol&Rampersad&Shallit:2020}.
In each case, the common prefix $p$ appears only as the prefix of the image of a letter.
This ensures that the morphism is synchronizing.
Then we check that the image of every considered Dejean word $u$ of length $t$ is $RT'(k)^+$-free,
where $t$ is specified by~\cite[Lemma~2.1]{Ochem:2006}.

In addition, using depth-first search of the
appropriate space, we have constructed:
\begin{itemize}
    \item A $(5/4)^+$-free word over 6 letters with pseudoperiod (9,24), of length 500000.
    \item A $(6/5)^+$-free word over 7 letters with pseudoperiod (22,33), of length 500000.
\end{itemize}

These examples suggest the following conjecture.
\begin{conjecture}
For every $k\ge4$ we have $RT'(k)=\tfrac{k-1}{k-2}$.
\label{conj2}
\end{conjecture}

\section{Computational complexity}
\label{sec7}

For a finite word, checking a given specific pseudoperiod is obviously easy.
However, checking the existence of an arbitrary pseudoperiod is computationally hard, as we show now.

Consider the following decision problem:\\

\noindent\pseudoperiod:\\
\noindent\emph{Instance:}  a string $x$ of length $n$,  and positive integers $k$ and $B$.

\noindent \emph{Question:}  Does there exist
a set $S = \{p_1, p_2, \ldots, p_k\}$ of cardinality $k$ with
$1 \leq p_1 < \cdots < p_k \leq B$ such that
$x[i] \in \{ x[i+p_1], x[i+p_2], \ldots, x[i+p_k] \}$ for
$1 \leq i \leq n-p_k$?

\begin{theorem}
\pseudoperiod\ is \mbox{\bf NP}-complete.
\end{theorem}

\begin{proof}
It is easy to see that \pseudoperiod\ is in \textbf{NP},
as we can check an instance in polynomial time.

To see that \pseudoperiod\ is \textbf{NP}-hard,
we reduce from a classical
\textbf{NP}-complete problem, namely, \hittingset\  
\cite{Karp:1972}.
It is defined as follows:\\

\noindent\hittingset\\
\noindent\emph{Instance:} 
A list of sets $S_1, S_2, \ldots, S_m$ over a universe $U = \{1, 2, \ldots, n\}$ and an integer $k'$.\\
\noindent\emph{Question:}
Does there exist a set $H = \{h_1, h_2, \ldots, h_{k'}\}$ of cardinality $k'$ such that $S_i \cap H \neq \emptyset$ for all $i$?\\

Given an instance of \hittingset\ $S_1, S_2, \ldots, S_m$ and $U = \{1, 2, \ldots, n\}$, and integer $k'$, define $a_{\ell, i} = 1$ if $\ell \in S_i$ and $a_{\ell, i} = 0$ otherwise.
We construct a \pseudoperiod\ instance with $k = k'+4$, $B = 4n+5$, and a string $x$, as follows:
\begin{align*}
x = uvwz_1z_2\cdots z_m
\end{align*}
and
\begin{align*}
u &= 11 \,0^{4n+3}\\
v &= 101 \, 0^{4n+3}\\
w &= 1001 \, 0^{4n+3}\\
z_i &= 1000 \, a_{1, i} 000 \, a_{2, i}\cdots 000 \, a_{n, i}0000 \, (0011)^n \, 0.
\end{align*}

We first show that if the \pseudoperiod\ instance has a solution, then we can extract a solution for the \hittingset\ instance.
To do so, we examine what a valid pseudoperiod would look like for $x$ by first considering each $1$ symbol.

The first $1$ symbol, $u[1] = x[1] = 1$, is followed by $10^{4n+3}$ and since the $p_j$ forming the pseudoperiod are bounded by $B = 4n+5$, we require $p_1 = 1$ for the pseudoperiod property to be satisfied at $x[1]$.
Similarly, the next $1$ symbol $u[2] = x[2] = 1$ is followed by $0^{4n+3}1$, which requires that some $p_j$ equal $4n+4$, in order to satisfy the pseudoperiod property.
The $v$ factor is analogous in that $v[1] = 1$ is followed by $010^{4n+3}$, which gives us that $p_2 = 2$ and $v[3] = 1$ has a $1$ symbol $4n+4$ symbols afterward, so it also satisfies the pseudoperiod property.
The $w$ factor is such that $w[1] = 1$ is followed by $0010^{4n+3}$, which then forces $p_3 = 3$ with $w[4]$ satisfied by having a $1$ symbol $4n+4$ symbols afterward as previous.

We now consider the $z_i$ factors.
For each $z_i$, the $1$ symbol at $z_i[1]$ satisfies the pseudoperiod property if and only if the pseudoperiod contains some $p_j$ such that $\frac{p_j}{4} \in S_i$.
Since the only possible indices that can be $1$ within the $B$ bound are the $a_{\ell, i}$, the pseudoperiod property is satisfied at $z_i[1]$ using some $p_j$ of the pseudoperiod if and only if $z_i[1+p_j] = a_{\frac{p_j}{4}, i} = 1$ which means $\frac{p_j}{4} \in S_i$.

Considering the remaining $1$ symbols in $z_i$, we see that each $a_{j, i}$ is followed by a $0$ symbol and has a $1$ symbol exactly $4n+4$ indices later in the $(0011)^n$ factor. Regardless of the assignment of $a_{j, i}$ the pseudoperiod property is satisfied.
Each $1$ symbol in the $(0011)^n$ factor has another $1$ symbol either $p_1 = 1$, $p_2 = 2$, or $p_3 = 3$ indices later, as each $0011$ is followed by one of: another $0011$, $01$ where the $1$ symbol is $z_{i+1}[1]$, or the end of the string if $i = m$ in which case it satisfies the pseudoperiod property by default.

Finally, we observe that there are no more than two consecutive $1$ symbols in $x$, so the pseudoperiod property is satisfied at every $0$ symbol, as there is another $0$ symbol either $p_1 = 1$, $p_2 = 2$, or $p_3 = 3$ indices later.

Taken together, a satisfying pseudoperiod for this instance is of the form $\{1, 2, 3, 4n+4\} \cup P$, where $P$ is a set of cardinality $k'$ that has the property for all $S_i$, there exists $p_j \in P$ such that $\frac{p_j}{4} \in S_i$.
Therefore, if such a pseudoperiod exists, then we can derive a solution $H = \{\frac{p}{4} \mid p \in P\}$ for the \hittingset\ instance from the solution to the generated \pseudoperiod\ instance.

Conversely, if the \hittingset\ instance has a solution $H$ then 
$$P = \{1, 2, 3, 4n+4\} \cup \{4 \cdot h \mid h \in H\}$$
is a valid pseudoperiod for $x$.
All of the $0$ symbols and most of the $1$ symbols are satisfied by the $p_1 = 1, p_2 = 2, p_3 = 3$, or $p_{k} = 4n+4$ as previously explained.
We only need to check that the $z_i[1] = 1$ also satisfy the desired property.
There exists some $h_i \in S_i \cap H$, since $H$ is a hitting set, which means that $a_{h_i, i} = 1$.
This gives us that $z_i[1 + h_i \cdot 4] = a_{h_i, i} = 1$ and $4 \cdot h_i \in P$, which means each $z_i[1]$ also satisfies the pseudoperiod property and $P$ is a pseudoperiod for this instance.

Therefore, \pseudoperiod\ is \textbf{NP}-Hard.   This completes the proof.
\end{proof}

\section{About Vladimir Shevelev}
\label{sec8}

Here we present some details about Vladimir Shevelev's life
and contributions, based on \cite{Shevelev:2022}.

Vladimir Samuil Shevelev was born 
on March 9 1945 in Novocherkassk, Russia,
under the name Vladimir Abramovich.
He received his Ph.D. in mathematics in 1971 from the Rostov-on-Don State University in the USSR.
In 1992 he received a D.Sc. in combinatorics from
the Glushkov Cybernetic Institute, Academy of Ukraine, Kiev.

From 1971 to 1974 he was Assistant Professor at the Department of Mathematics, Rostov State University.  From 1974 to 1999 he taught at the Department of Mathematics, Rostov State Building University.
In 1982 he took the surname ``Shevelev'' and in 1999 he emigrated to Israel, where he taught at the Ben-Gurion University of the Negev and did research at the Tel Aviv University.

From 1969 to 2016, Vladimir Shevelev published approximately 60 mathematical papers in refereed journals.  He also published approximately 40 preprints on the arXiv.  He was an excellent chess player, played the violin, and was a member of a Russian vocal group.  He was married and had three children and six grandchildren.     He died on May 3 2018 in Beersheba, Israel.

May his memory be a blessing.
\begin{center}
    \includegraphics[width=2.3in]{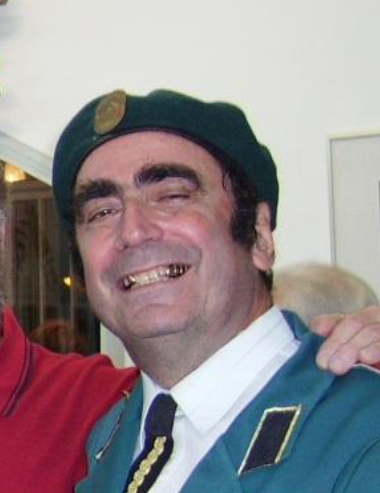}
\end{center}
Photograph taken from \url{https://www.math.bgu.ac.il/~shevelev/Hobbies.pdf}.

\acknowledgements
This work benefited from the use of the CrySP RIPPLE Facility at the University of Waterloo. Thanks to Ian Goldberg for allowing us to run computations on this machine.

We thank the referees for a careful reading of the paper and for many useful suggestions and corrections.

We are grateful to Daniel Berend and Simon Litsyn for their assistance in obtaining information about the life of Vladimir Shevelev.

\appendix
\begin{appendices}
\section{The automaton \texttt{triple}}

In this section we provide the \texttt{Walnut} code for the automaton
\texttt{triple}.

\bigskip

\begin{minipage}{1in}
{\tiny
\begin{verbatim}
msd_2 msd_2 msd_2

0 0
0 0 0 -> 0
0 0 1 -> 1
0 1 1 -> 2
1 1 1 -> 3

1 0
0 1 0 -> 4
1 1 0 -> 5
0 1 1 -> 6
1 1 1 -> 7

2 0
0 0 0 -> 8
1 0 0 -> 9
0 0 1 -> 10
1 0 1 -> 11
1 1 1 -> 12

3 0
0 0 0 -> 3
0 0 1 -> 13
0 1 1 -> 14
1 1 1 -> 3

4 0
0 0 0 -> 15
1 0 0 -> 16
0 1 0 -> 17
1 1 0 -> 18
1 1 1 -> 19

5 0
0 0 0 -> 20
0 1 0 -> 21
1 1 0 -> 13
0 1 1 -> 22

6 0
1 1 0 -> 23
\end{verbatim}
}
\end{minipage}
\begin{minipage}{1in}
{\tiny
\begin{verbatim}
7 0
0 1 1 -> 24
1 1 1 -> 25

8 0
0 0 0 -> 26
1 0 0 -> 27

9 0
0 0 0 -> 28
1 0 0 -> 14
0 0 1 -> 29
1 0 1 -> 19

10 0
1 1 0 -> 22

11 1
0 0 0 -> 30
1 1 0 -> 24
1 0 1 -> 31
1 1 1 -> 19

12 0
0 0 1 -> 24
1 0 1 -> 32
1 1 1 -> 33

13 0
0 1 0 -> 19
1 1 0 -> 13

14 0
1 0 0 -> 14
1 0 1 -> 19

15 0
0 0 0 -> 15
1 1 1 -> 19

16 0
1 0 0 -> 34
\end{verbatim}
}
\end{minipage}
\begin{minipage}{1in}
{\tiny
\begin{verbatim}
17 0
0 1 0 -> 35
0 1 1 -> 6

18 0
1 1 0 -> 18
1 1 1 -> 24

19 1
0 0 0 -> 19
1 1 1 -> 19

20 0
0 0 0 -> 36
0 1 0 -> 37

21 1
0 0 0 -> 38
0 1 0 -> 39
1 0 1 -> 24
1 1 1 -> 19

22 0
1 0 0 -> 24

23 0
1 1 0 -> 40

24 1
0 0 0 -> 24

25 0
0 1 0 -> 24
0 1 1 -> 24
1 1 1 -> 7

26 0
0 0 0 -> 41
0 0 1 -> 10
1 0 1 -> 42
\end{verbatim}
}
\end{minipage}
\begin{minipage}{1in}
{\tiny
\begin{verbatim}
27 0
1 0 1 -> 43

28 0
0 0 0 -> 44
0 0 1 -> 29
1 0 1 -> 45

29 0
1 1 0 -> 24

30 1
0 0 0 -> 46
1 1 1 -> 19

31 0
1 0 1 -> 47

32 1
0 0 0 -> 24
1 0 1 -> 48

33 0
0 0 1 -> 24
1 0 1 -> 24
1 1 1 -> 33

34 0
1 0 0 -> 16
1 1 1 -> 24

35 0
0 1 0 -> 17
1 1 0 -> 49

36 0
0 0 0 -> 20
0 1 0 -> 37
0 1 1 -> 22
\end{verbatim}
}
\end{minipage}
\begin{minipage}{1in}
{\tiny
\begin{verbatim}
37 0
1 0 1 -> 24

38 1
0 0 0 -> 50
1 1 1 -> 19

39 0
0 1 0 -> 51

40 1
0 0 0 -> 24
1 1 0 -> 23

41 0
0 0 0 -> 26

42 1
0 0 0 -> 52
1 1 0 -> 24

43 0
0 1 1 -> 24

44 0
0 0 0 -> 28
0 0 1 -> 29

45 0
0 1 0 -> 24

46 1
0 0 0 -> 30
1 1 0 -> 24
1 1 1 -> 19

47 0
1 0 1 -> 31
1 1 1 -> 24
\end{verbatim}
}
\end{minipage}
\begin{minipage}{1in}
{\tiny
\begin{verbatim}
48 0
1 0 1 -> 32

49 0
1 1 1 -> 24

50 1
0 0 0 -> 38
1 0 1 -> 24
1 1 1 -> 19

51 0
0 1 0 -> 39
1 1 1 -> 24

52 1
0 0 0 -> 42
\end{verbatim}
}
\end{minipage}

\end{appendices}

\end{document}